\documentclass[11pt,a4paper]{article}

\usepackage{inputenc}
\usepackage{amsmath}
\usepackage{bm}
\usepackage{bbold}
\usepackage{amsthm}

\usepackage{curves}
\usepackage{epic}

\usepackage{hyperref}

\title{Evaluation of Lyapunov Exponent in Generalized Linear Dynamical Models of Queueing Networks\thanks{Proc. MATHMOD 09 Vienna Full Papers CD Volume (I.~Troch, F.~Breitenecker, eds.). ARGESIM, Vienna, 2009, pp.~706--717.}} 

\author{N.~K.~Krivulin\thanks{Faculty of Mathematics and Mechanics, St.~Petersburg State University, 28 Universitetsky Ave., St.~Petersburg, 198504, Russia, 
nkk@math.spbu.ru.}
}

\date{}

\newtheorem{theorem}{Theorem}
\newtheorem{lemma}[theorem]{Lemma}

\begin{document}

\maketitle

\begin{abstract}
The problem of evaluation of Lyapunov exponent in queueing network analysis is considered based on models and methods of idempotent algebra. General existence conditions for Lyapunov exponent to exist in generalized linear stochastic dynamic systems are given, and examples of evaluation of the exponent for systems with matrices of particular types are presented. A method which allow one to get the exponent is proposed based on some appropriate decomposition of the system matrix. A general approach to modeling of a wide class of queueing networks is taken to provide for models in the form of stochastic dynamic systems. It is shown how to find the mean service cycle time for the networks through the evaluation of Lyapunov exponent for their associated dynamic systems. As an illustration, the mean service time is evaluated for some systems including open and closed tandem queues with finite and infinite buffers, fork-join networks, and systems with round-robin routing.
\end{abstract}

\section{Introduction}

The evolution of actual systems encountered in economics, management, engineering, and other areas can frequently be represented through stochastic dynamic models of the form
$$
\bm{x}(k)=A(k)\bm{x}(k-1),
$$
where $ A(k) $ is a random state transition matrix, $ \bm{x}(k) $ is a state vector, and matrix-vector multiplication is thought of as defined in terms of a semiring with the operations of taking maximum and addition \cite{Cohen1988Subadditivity,Baccelli1993Synchronization,Kolokoltsov1997Idempotent,Litvinov1998Idempotent,Heidergott2006Maxplus}.

In many cases, the analysis of a system involves evaluation of asymptotic growth rate of the system state vector $ \bm{x}(k) $, which is normally referred to as Lyapunov exponent \cite{Baccelli1993Synchronization,Jean-Marie1994Analytical}. In the semiring with the operations of maximum and addition, the mean growth rate is defined in terms of the semiring as
$$
\lambda=\lim_{k\to\infty}\|\bm{x}(k)\|^{1/k}.
$$

Note that for queueing networks, the value of $ \lambda $ can be considered as the mean service cycle time in a system, while its inverse can be thought of as system throughput.

Existence conditions for the above limit can normally be established based on the ergodic theorem in \cite{Kingman1973Subadditive}. Specifically, examples of the conditions can be found in \cite{Cohen1988Subadditivity,Baccelli1993Synchronization,Glasserman1995Stochastic}). Note however, that these conditions frequently require that the system state transition matrix is of particular form or type (matrix with finite entries, irreducible matrix).

For systems with fixed nonrandom matrices, the value of $ \lambda $ can easily be found based on results in \cite{Romanovskii1967Optimization,Romanovskii1967Asymptotic} (see also \cite{Kolokoltsov1997Idempotent, Krivulin2007Ontheconvergence}. However, for stochastic systems with random state transition matrices, evaluation of $ \lambda $ normally appears to be a difficult problem. Existing results \cite{Cohen1988Subadditivity,Olsder1990Discrete,Jean-Marie1994Analytical,Krivulin2007Growth,Krivulin2008Evaluation} include the exact solution to the problem for second-order systems determined by matrices with entries that are independent and have either normal or exponential probability distributions. It is shown in \cite{Krivulin2001Onevaluation, Krivulin2002Evaluation,Krivulin2003Estimation, Krivulin2005Thegrowth} how the value of $ \lambda $ can be found for systems with particular matrices of arbitrary size, including triangular matrices, similarity matrices, and matrices of rank 1.

In the current paper we consider the problem of evaluation of Lyapunov exponent, which arises in the analysis of a wide class of queueing networks. We start with an overview of some concepts and results in idempotent algebra that underlie subsequent parts, and introduce related notations.

We consider a general model of stochastic dynamical system governed by the vector equation with the random matrices $ A(k) $ which are assumed to be independent and identically distributed for all $ k=1,2,\ldots $ New general conditions for the limit that determines the Lyapunov exponent $ \lambda $ to exist are given.

We present results of evaluating $ \lambda $ for systems with matrices of particular types, including diagonal and triangular matrices, similarity matrices, and matrices of rank $ 1 $. A new approach to evaluation of $ \lambda $ is then proposed based on a decomposition of $ A(k) $. The approach allows one to reduce the problem with the original matrix $ A(k) $ to that with a matrix of particular type and known solution.

Furthermore, we discuss application of the above results to the analysis of a class of queueing systems described in \cite{Krivulin1994Using,Krivulin1995Amax-algebra,Krivulin1996Analgebraic,Krivulin1996Max-plus}. Some particular systems including open and closed tandem queues with finite and infinite buffers, a fork-join network, and a system with round-robin routing are examined. For these queueing systems, we obtain the mean service cycle time through evaluation of Lyapunov exponent in their associated dynamical systems.

\section{Preliminary results}

\subsection{Idempotent algebra}

Consider a set $ \mathbb{X} $ with two operations $ \oplus $ and $ \otimes $ referred to as addition and multiplication, and neutral elements $ \mathbb{0} $ and $ \mathbb{1} $ called zero and identity. We suppose that $ (\mathbb{X},\oplus,\otimes) $ is a commutative semiring with idempotent addition and invertible multiplication. Such a semiring is normally called an idempotent semifield.

Let $ \mathbb{X}_{+}=\mathbb{X}\setminus\{\mathbb{0}\} $. Then each $ x\in\mathbb{X}_{+} $ is assumed to have the inverse $ x^{-1} $. For any $ x,y\in\mathbb{X}_{+} $, the power $ x^{y} $ is defined in an ordinary way. Note that in what follows, the sign $ \otimes $ will be omitted as is usual in conventional algebra. The notation of power is thought of as defined in terms of idempotent algebra. However, for the sake of simplicity, we use ordinary arithmetic operations in the expressions that represent exponents.

Since the addition is idempotent, one can define a linear order $ \leq $ on $ \mathbb{X} $ according to the rule: $ x\leq y $ if and only if $ x\oplus y=y $. Below the relation symbols $ \leq $ will be understood only in the sense of this linear order. According to the order, for any $ x\in\mathbb{X} $, we have $ x\geq\mathbb{0} $.

Examples of idempotent semirings (semifields) include
\begin{align*}
\mathbb{R}_{\max,+}&=(\mathbb{R}\cup\{-\infty\},\max,+),
&\mathbb{R}_{\min,+}&=(\mathbb{R}\cup\{+\infty\},\min,+), \\
\mathbb{R}_{\max,\times}&=(\mathbb{R}_{+}\cup\{0\},\max,\times),
&\mathbb{R}_{\min,\times}&=(\mathbb{R}_{+}\cup\{+\infty\},\min,\times),
\end{align*}
where $ \mathbb{R} $ is the set of real numbers, $ \mathbb{R}_{+}=\{x\in\mathbb{R}|x>0\} $.

Consider the semiring $ \mathbb{R}_{\max,+} $. It is easy to see that $ \mathbb{0}=-\infty $, $ \mathbb{1}=0 $. For each $ x\in\mathbb{R} $, there exists its inverse $ x^{-1} $, which is equal to $ -x $ in the ordinary arithmetics. For any $ x,y\in\mathbb{R} $, the power $ x^{y} $ coincides with the arithmetic product $ xy $. The linear order has conventional meaning.

For the semiring $ \mathbb{R}_{\min,\times} $, we have $ \mathbb{0}=+\infty $, $ \mathbb{1}=1 $. The inverse element and the power have ordinary meaning. The relation $ \leq $ determines an order that is the reverse of the natural linear order on $ \mathbb{R}_{+} $.

It is easy to verify that the semirings $ \mathbb{R}_{\max,+} $, $ \mathbb{R}_{\min,+} $, $ \mathbb{R}_{\max,\times} $, and $ \mathbb{R}_{\min,\times} $ are all isomorphic.

\subsection{Matrix algebra}

For any matrices $ A,B\in\mathbb{X}^{m\times n} $, $ C\in\mathbb{X}^{n\times l} $, and a scalar $ x\in\mathbb{X} $, the matrices $ A\oplus B $, $ BC $, and $ xA $ are defined in a usual way with the formulas
$$
\{A\oplus B\}_{ij}=\{A\}_{ij}\oplus\{B\}_{ij},
\quad
\{B C\}_{ij}=\bigoplus_{k=1}^{n}\{B\}_{ik}\{C\}_{kj},
\quad
\{x A\}_{ij}=x\{A\}_{ij}.
$$

The matrix with all its entries equal to $ \mathbb{0} $ is referred to as zero matrix and denoted by $ \mathbb{0} $. A matrix is called regular if every row of the matrix has at least one nonzero entry. Below we assume all nonzero matrix to be regular.

For each matrix $ A=(a_{ij})\in\mathbb{X}^{m\times n} $, its norm is defined as
$$
\|A\|=\bigoplus_{i=1}^{n}\bigoplus_{j=1}^{m}a_{ij}.
$$

For any matrices $ A $ and $ B $ of appropriate size, and a scalar $ x\geq\mathbb{1} $, it holds
$$
\|A\oplus B\|
=
\|A\|\oplus\|B\|,
\qquad
\|AB\|
\leq
\|A\|\|B\|,
\qquad
\|xA\|
=
x\|A\|.
$$

A square matrix is called diagonal if all its off-diagonal entries are zero, and triangular if its entries either above or below the diagonal are zero. The matrix $ I=\mathop\mathrm{diag}\nolimits(\mathbb{1},\ldots,\mathbb{1}) $ is referred to as identity matrix. The integer nonnegative powers of a matrix $ A $ are defined by the formulas $ A^{0}=I $, $ A^{k+l}=A^{k}A^{l} $ for all $ k,l=0,1,\ldots $

A matrix $ A\in\mathbb{X}^{n\times n} $ is a matrix of a similarity operator (similarity matrix) if there exists a constant $ \alpha\ne\mathbb{0} $ called the coefficient of similarity, such that for any $ \bm{x}\in\mathbb{X}^{n} $, it holds
$$
\|A\bm{x}\|=\alpha\|\bm{x}\|.
$$

A vector $ \bm{y}\in\mathbb{X}^{n} $ is said to be linearly dependent of $ \bm{x}_{1},\ldots,\bm{x}_{m}\in\mathbb{X}^{n} $ if $ \bm{y}=a_{1}\bm{x}_{1}\oplus\cdots\oplus a_{m}\bm{x}_{m} $ for some scalars $ a_{1},\ldots,a_{m}\in\mathbb{X} $. The rank of a matrix $ A $ is its maximum number of linearly independent rows (columns). A matrix $ A $ has the rank $ 1 $ if and only if $ A=\bm{x}\bm{y}^{T} $, where $ \bm{x} $ and $ \bm{y} $ are some nonzero vectors.

The sum of diagonal entries of a matrix $ A=(a_{ij})\in\mathbb{X}^{n\times n} $ is called its trace, and denoted by
$$
\mathop\mathrm{tr}\nolimits A=\bigoplus_{i=1}^{n}a_{ii}.
$$

A number $ \lambda $ is called an eigenvalue of the matrix $ A $ if there exists a vector $ \bm{x}\ne\mathbb{0} $ such that $ A\bm{x}=\lambda\bm{x} $.

The maximum (in the sense of the linear order on $ \mathbb{X} $) eigenvalue is called the spectral radius of $ A $ and evaluated as
$$
\rho(A)
=
\bigoplus_{m=1}^{n}\mathop\mathrm{tr}\nolimits^{1/m}(A^{m}).
$$

The next result has been obtained in \cite{Romanovskii1967Optimization, Romanovskii1967Asymptotic} (see also \cite{Kolokoltsov1997Idempotent, Krivulin2007Ontheconvergence}).
\begin{theorem}[Romanovskii]\label{T-limAkl}
For any matrix $ A\in\mathbb{X}^{n\times n} $ there exist limits
$$
\lim_{k\to\infty}\|A^{k}\|^{1/k}
=
\rho(A),
\qquad
\lim_{k\to\infty}\mathop\mathrm{tr}\nolimits^{1/k}(A^{k})
=
\rho(A).
$$
\end{theorem}

\subsection{Properties of the expected value}

Consider some inequalities which relate evaluation of the expected value of random variables to operations involved in idempotent semirings. To be definite, assume that the semiring of interest is $ \mathbb{R}_{\max,+} $.

We suppose that all random variables under considerations are defined on a common probability space and have finite mean values. 

It is easy to verify that for any random variables $ \xi $ and $ \eta $ it holds
$$
\mathsf{E}(\xi\oplus\eta)
\geq
\mathsf{E}\xi\oplus\mathsf{E}\eta,
\qquad
\mathsf{E}\xi\eta
=
\mathsf{E}\xi\mathsf{E}\eta.
$$

Let $ A $ be a random matrix. We use the symbol $ \mathsf{E}A $ to denote the matrix obtained from $ A $ by replacing all its entries with their expected values, provided that $ \mathsf{E}\mathbb{0}=\mathbb{0} $.

For any random matrices $ A $ and $ B $ of appropriate size, it holds
$$
\mathsf{E}(A\oplus B)
\geq
\mathsf{E}A\oplus\mathsf{E}B,
\quad
\mathsf{E}AB
\geq
\mathsf{E}A\mathsf{E}B,
\quad
\mathsf{E}\|A\|
\geq
\|\mathsf{E}A\|,
\quad
\mathsf{E}\mathop\mathrm{tr}\nolimits A
\geq
\mathop\mathrm{tr}\nolimits(\mathsf{E}A).
$$

\section{Stochastic dynamical systems}

Let $ A(k)\in\mathbb{X}^{n\times n} $ be a random state transition matrix, $ \bm{x}(k)\in\mathbb{X}^{n} $ be a state vector, $ k=1,2,\ldots $ Consider a dynamical system described by the equation
$$
\bm{x}(k)
=
A^{T}(k)\bm{x}(k-1).
$$

Note that the representation in the form with the transpose is intended to simplify further formulae. Clearly, the use of the matrix $ A^{T}(k) $ does not actually change the general form of the equation which is normally written as $ \bm{x}(k)=A(k)\bm{x}(k-1) $.

In particular, the above representation allows products of matrices $ A(1),\ldots,A(k) $ to be examined in the natural order. With the notation
$$
A_{k}
=
A(1)\cdots A(k),
$$
iterating the dynamic equation gives $ \bm{x}(k)=A_{k}^{T}\bm{x}(0) $.

Suppose that the sequence $ \{A(k) | k\geq1\} $ consists of independent and identically distributed random matrices, $ \mathsf{E}\|A_{1}\| $ is finite. However, we do not require that for each $ k=1,2,\ldots $, the entries of $ A(k) $ are independent.

Note that most of the statements below remain valid if the matrices $ A(k) $ form a stationary sequence rather than a sequence of independent and identically distributed matrices.

\subsection{Lyapunov exponent}

In many cases, the analysis of a system involves evaluation of asymptotic growth rate of the system state vector $ \bm{x}(k) $, which is frequently referred to as Lyapunov exponent. In the semirings $ \mathbb{R}_{\max,+} $ and $ \mathbb{R}_{\min,+} $, the mean growth rate is defined as
\begin{equation}\label{E-lambdax}
\lambda=\lim_{k\to\infty}\|\bm{x}(k)\|^{1/k}
\end{equation}
provided that the limit on the right side exists.

In $ \mathbb{R}_{\max,\times} $ and $ \mathbb{R}_{\min,\times} $, we have
$$
\lambda
=
\lim_{k\to\infty}\log\|\bm{x}(k)\|^{1/k}.
$$

In what follows, we will concentrate on the semiring $ \mathbb{R}_{\max,+} $. Taking into account isomorphisms between $ \mathbb{R}_{\max,+} $ and the semirings $ \mathbb{R}_{\min,+} $, $ \mathbb{R}_{\max,\times} $ and $ \mathbb{R}_{\min,\times} $, all results below can easily be extended to these semirings.

Suppose that the entries of the initial vector $ \bm{x}(0) $ are finite with probability one (w.~p.~1). Then the mean growth rate $ \lambda $ can be defined as
\begin{equation}\label{E-lambdaA}
\lambda
=
\lim_{k\to\infty}
\|A_{k}\|^{1/k}.
\end{equation}

\subsection{Existence conditions}

General conditions for limit \eqref{E-lambdaA} to exist can be established based on the classical result in \cite{Kingman1973Subadditive}, which is presented below in terms of ordinary arithmetic operations.
\begin{theorem}[Kingman]\label{T-King}
Let $ \{\zeta_{lm} | l<m \} $ be a family of random variables which satisfy the following properties:
\begin{enumerate}
\item $ \zeta_{lm}\leq\zeta_{lk}+\zeta_{km} $ (subadditivity);
\item The joint distributions are the same for both families $ \{\zeta_{lm} | l<m \} $ and $ \{\zeta_{l+1,m+1}| l<m \} $ (stationarity);
\item For all $ n=1,2,\ldots $, there exists $ \mathsf{E}\zeta_{0k}\geq -ck $ for some positive constant $ c $ (boundedness).
\end{enumerate}

Then there exists a constant $ \lambda $, such that it holds
$$
\lim_{k\to\infty}\zeta_{0k}/k=\lambda
\quad
\text{w.~p.~1},
\qquad
\lim_{k\to\infty}\mathsf{E}\zeta_{0k}/k=\lambda.
$$
\end{theorem}

For the semiring $ \mathbb{R}_{\max,+} $, the existence conditions can be formulated as follows.
\begin{theorem}\label{T-King1}
Let $ \{A(k) | k\geq1\} $ be a stationary sequence of random matrices, $ \mathsf{E}\|A_{1}\|<\infty $ and $ \rho(\mathsf{E}[A_{1}])>-\infty $. Then there exists a finite number $ \lambda $ such that
$$
\lim_{k\to\infty}\|A_{k}\|^{1/k}=\lambda
\quad
\mbox{w.~p.~1},
\qquad
\lim_{k\to\infty}\mathsf{E}\|A_{k}\|^{1/k}=\lambda.
$$
\end{theorem}
\begin{proof}
Consider the family $ \{\zeta_{lm} | l<m \} $ with $ \zeta_{lm}=\|A(l+1)A(l+2)\cdots A(m)\| $, and note that $ \zeta_{0k}=\|A_{k}\| $.

Let us verify that the conditions of Theorem~\ref{T-King} are fulfilled.

Subadditivity follows from the inequality
$$
\|A(l+1)\cdots A(m)\|
\leq
\|A(l+1)\cdots A(k)\|\|A(k+1)\cdots A(m)\|.
$$

It is clear that the family is stationary. From the condition $ \mathsf{E}\|A_{1}\|<\infty $ it follows that $ \mathsf{E}\|A_{k}\|<\infty $. Furthermore, since $ \rho(\mathsf{E}A_{1})>\mathbb{0}=-\infty $ and $ \mathsf{E}\|A_{k}\|=\mathsf{E}\|A(1)\cdots A(k)\|\geq\|(\mathsf{E}A_{1})^{k}\|\geq\rho^{k}(\mathsf{E}A_{1}) $, the family is bounded.

By applying Theorem~\ref{T-King}, we arrive at the desired result.
\end{proof}

\section{Evaluation of Lyapunov exponent}

Now we give examples of evaluation of Lyapunov exponent for systems with matrices of particular types \cite{Krivulin2001Onevaluation, Krivulin2002Evaluation,Krivulin2003Estimation, Krivulin2005Thegrowth}. Note that for each system under consideration, the conditions of Theorem~\ref{T-King1} are assumed to be fulfilled.

\subsection{Systems with diagonal matrix}

Suppose that $ A(k)=\mathop\mathrm{diag}\nolimits(d_{1}(k),\ldots,d_{n}(k)) $ is a diagonal matrix, $ k=1,2,\ldots $ It is easy to verify that
$$
\lambda
=
\mathop\mathrm{tr}\nolimits(\mathsf{E}A_{1}).
$$

\subsection{Systems with similarity matrix}

Let $ A(k) $ be the matrix of a similarity operator, $ k=1,2,\ldots $ Then it holds
$$
\lambda
=
\mathsf{E}\|A_{1}\|.
$$

\subsection{Systems with matrix of rank $ 1 $}

Suppose that for all $ k=1,2,\ldots $, there exist vectors $ \bm{u}(k) $ and $ \bm{v}(k) $ such that $ A(k)=\bm{u}(k)\bm{v}^{T}(k) $. Then it holds
$$
\lambda
=
\mathsf{E}[\bm{v}^{T}(1)\bm{u}(2)].
$$

\subsection{Systems with triangular matrix}

The following result was obtained in \cite{Krivulin2005Thegrowth} based on classical bounds for the mean values of sums of independent random variables combined with new inequalities for products of triangular matrices in idempotent algebra.

\begin{theorem}\label{T-lambdaTr}
If the matrix $ A(k) $ is triangular $ k=1,2,\ldots $, then it holds
$$
\lambda
=
\mathop\mathrm{tr}\nolimits(\mathsf{E}A_{1}).
$$
\end{theorem}

\subsection{A matrix decomposition method}

If the state transition matrix of a system falls into the one of the particular types examined above, evaluation of Lyapunov exponent presents no special problems. In the case that the matrix has a different type, one can try to implement the following approach.

Let us assume that there exists a decomposition of the matrix $ A (k) $ in the form
$$
A(k)
=
B(k)C(k),
$$
where $ B(k) $ and $ C(k) $ are some independent matrices. If a solution to the problem of evaluation of Lyapunov exponent for the system with the matrix
$$
A^{\prime}(k)
=
C(k)B(k+1)
$$
is known, it can normally be taken as the solution for the system with the initial matrix $ A(k) $. Otherwise, one can continue with decomposition of $ A^{\prime}(k) $.

Note that for any matrix $ A(k) $ of rank $ 1 $, the above decomposition exists and takes the form $ A(k)=\bm{u}(k)\bm{v}^{T}(k) $.

\subsection{A system with a matrix of incomplete rank}

Consider a matrix $ A $ of order $ n $. Suppose that $ \mathop\mathrm{rank}\nolimits A<n $. Then there exits a skeleton decomposition of the matrix
$$
A
=
BC.
$$
The above decomposition will be referred to as backward triangular if the matrix $ CB $ is triangular.

\begin{theorem}\label{T-lambdaC1B2}
If the matrix $ A(k) $ allows for the backward triangular skeleton decomposition $ A(k)=B(k)C(k) $ with independent factors $ B(k) $ and $ C(k) $, then
$$
\lambda
=
\mathop\mathrm{tr}\nolimits(\mathsf{E}[C(1)B(1)]).
$$
\end{theorem}
\begin{proof}
Consider the matrix
$$
A_{k}
=
\prod_{j=1}^{k}B(j)C(j)
=
B(1)\left(\prod_{j=1}^{k-1}C(j)B(j+1)\right)C(k).
$$

Furthermore, we have
\begin{multline*}
\mathsf{E}\|A_{k}\|
\geq
\mathsf{E}\mathop\mathrm{tr} A_{k}
\geq
\mathsf{E}\left[\mathop\mathrm{tr}\nolimits\left(C(k)B(1)\prod_{j=1}^{k-1}C(j)B(j+1)\right)\right]
\\
\geq
\mathop\mathrm{tr}\left(\mathsf{E}[C(1)B(1)]^{k}\right).
\end{multline*}

By applying Theorem~\ref{T-limAkl}, we get
$$
\lim_{k\to\infty}\mathsf{E}\|A_{k}\|^{1/k}
\geq
\lim_{k\to\infty}\left(\mathop\mathrm{tr}\nolimits(\mathsf{E}[C(1)B(1)]^{k})\right)^{1/k}
=
\rho(\mathsf{E}[C(1)B(1)]).
$$

Since $ C(1)B(1) $ is a triangular matrix, we have the inequality
$$
\lambda
\geq
\mathop\mathrm{tr}(\mathsf{E}[C(1)B(1)]).
$$

It remains to verify that the opposite inequality is also valid. First we write
$$
\|A_{k}\|
\leq
\left\|\prod_{j=1}^{k-1}C(j)B(j+1)\right\|\|B(1)\|\|C(k)\|.
$$

The matrices $ C(j)B(j+1) $ are triangular and independent for all $ j=1,\ldots,k-1 $. With Theorem~\ref{T-lambdaTr}, we get
$$
\lambda
\leq
\mathop\mathrm{tr}(\mathsf{E}[C(1)B(1)]).
\qed
$$
\renewcommand{\qed}{}
\end{proof}

\section{Algebraic models of queueing networks}

Queueing networks with fork-join operations present a quite general class of dynamical systems that can be described in terms of the semiring $ \mathbb{R}_{\max,+} $ by the equation \cite{Krivulin1994Using,Krivulin1995Amax-algebra,Krivulin1996Analgebraic,Krivulin1996Max-plus,Krivulin2001Algebraic}
\begin{equation}\label{E-xkAkxk1}
\bm{x}(k)
=
A(k)\bm{x}(k-1).
\end{equation}

The fork and join operations allow customers (jobs, tasks) to be split into parts, and to be merged into one, when they circulate through the system. The fork-join formalism proves to be useful in the description of dynamical processes in a variety of actual systems, including production processes in manufacturing, transmission of messages in communication networks, and parallel data processing in multi-processor computer systems. As an illustration of the fork and join operations, one can consider respectively splitting a message into packets in a communication network, each intended for
transmitting via separate paths, and merging the packets at a destination node of the network to restore the message \cite{Baccelli1989Queueing}.

\subsection{Fork-join queueing networks}

Consider a network with $ n $ single-server nodes and customers of a single class. The network topology is described by an oriented graph $ \mathcal{G}=(V,E) $, where $ V=\{1,\ldots,n\} $ is a set of nodes, and $ E=\{(i,j)\}\subset V\times V $ is a set of arcs that determine the transition routes of customers.

For every node $ i\in V $, we define the sets $ P(i)=\{j|(j,i)\in E\} $ and $ S(i)=\{j|(i,j)\in E\} $. The nodes $ i $ with $ P(i)=\emptyset $ are assumed to be source nodes that represent infinite external arrival streams of customers. If $ S(i)=\emptyset $, the node $ i $ is considered as an output node intended to release customers from the network.

Each node $ i $ includes a server and a buffer which together present a single-server queue operating under the first-come, first-served (FCFS) discipline. At the initial time, all servers are free of customers. The buffer at each source node has infinite number of customers, whereas the buffer at any other node $ i $ has $ c_{i} $ customers, $ 0\leq c_{i}<\infty $.

We suppose that in the network, the usual service procedure is combined with additional join and fork operations \cite{Baccelli1989Queueing} which may be performed in a node respectively before and after service of a customer. The join operation is actually thought to cause each customer which comes into node $ i $, not to enter the buffer at the server but to wait until at least one customer from every node $ j\in P(i) $ arrives. As soon as these customers arrive, they, taken one from each preceding node, are united to be treated as being one customer which then enters the buffer to become a new member of the queue.

The fork operation at node $ i $ is initiated every time the service of a customer is completed; it consists in giving rise to several new customers instead of the former one. As many new customers appear in node $ i $ as
there are succeeding nodes included in the set $ S(i) $. These customers simultaneously depart the node, each being passed to separate node $ j\in S(i) $. We assume that the execution of fork-join operations when appropriate customers are available, as well as the transition of customers within and between nodes require no time.

\subsection{Dynamical equation}

Let $ \tau_{ik} $ be the $k$th service time, and $ x_{i}(k) $ be the $k$th departure epoch at node $ i=1,\ldots,n $.

We assume that for all $ i=1,\ldots,n $, the sequence $ \{\tau_{ik}, k\geq1\} $ consists of independent and identically distributed nonnegative random variables with a finite mean value. Furthermore, for each $ k $, the random variables $ \tau_{ik} $ and $ \tau_{jk} $ are independent for all $ i\ne j $.

With the condition that the network starts operating at time zero, we put $ x_{i}(0)=0=\mathbb{1} $ and $ x_{i}(k)=-\infty=\mathbb{0} $ for all $ k<0 $. We introduce notation
$$
\bm{x}(k)
=
\left(
	\begin{array}{c}
		x_{1}(k) \\
		\vdots \\
		x_{n}(k)
	\end{array}
\right),
\qquad
T_{k}
=
\left(
	\begin{array}{ccc}
		\tau_{1k}		& 				& \mathbb{0} \\
								& \ddots	& \\
		\mathbb{0}	& 				& \tau_{nk}
	\end{array}
\right).
$$

Let us define $ M=\max\{c_{i}| c_{i}<\infty, i=1,\ldots,n\} $. For each $ m=0,1,\ldots,M $, we introduce the matrix $ G_{m}=(g_{ij}^{m}) $ with entries
$$
g_{ij}^{m}
=
\begin{cases}
\mathbb{1},	& \text{if $ i\in P(j) $ and $ m=c_{j} $}, \\
\mathbb{0},	& \text{otherwise}.
\end{cases}
$$

It is easy to see that the matrix $ G_{m} $ can be considered as an adjacency matrix of the partial graph $ \mathcal{G}_{m}=(V,E_{m}) $, where $ E_{m}=\{(i,j)|i\in P(j), c_{j}=m\} $.

Now we can formulate the following result \cite{Krivulin1995Amax-algebra,Krivulin1996Max-plus,Krivulin2001Algebraic}.

\begin{lemma}
Suppose that the graph $ \mathcal{G}_{0} $ associated with the matrix $ G_{0} $ is acyclic, and $ r $ is the length of its longest path. Then the dynamics of the network is described in the semiring $ \mathbb{R}_{\max,+} $ by the equation
\begin{equation}
\bm{x}(k)
=
\bigoplus_{m=1}^{M}A_{m}(k)\bm{x}(k-m),\label{E-xk-Exp}
\end{equation}
where
\begin{align*}
A_{1}(k)
&=
(I\oplus T_{k}G_{0}^{T})^{r}T_{k}(I\oplus G_{1}^{T}), \\
A_{m}(k)
&=
(I\oplus T_{k}G_{0}^{T})^{r}T_{k}G_{m}^{T},
\quad
m=2,\ldots,M.
\end{align*}

\end{lemma}

\subsection{Networks with finite buffers}

Suppose now that the buffers at servers in the network may have limited capacity. In such a network, servers may be blocked according to some blocking mechanism \cite{Baccelli1989Queueing,Krivulin1995Amax-algebra}. Below we consider networks operating under the manufacturing and communication blocking rules which are commonly encountered in practice.

Manufacturing blocking at node $ i $ implies that a customer cannot release the server if there is at least one succeeding node $ j\in S(i) $ without empty space in its buffer. The communication blocking rule requires the server in node $ i $ not to initiate service of a customer until there is an empty space in the buffer in each node $ j \in S(i) $.

Suppose that the buffer at node $ i $ has capacity $ b_{i} $, $ 0\leq b_{i}\leq\infty $. Clearly, for all $ i=1,\ldots,n $, we have $ b_{i}\geq c_{i} $.

Let us define $ M_{1}=\max\{c_{i}| c_{i}<\infty, i=1,\ldots,n\} $ and $ M_{2}=\max\{b_{i}| b_{i}<\infty, i=1,\ldots,n\}+1 $. For each $ m=1,\ldots,M_{2} $, we introduce the matrix $ H_{m}=(h_{ij}^{m}) $ with its entries
$$
h_{ij}^{m}
=
\begin{cases}
\mathbb{1},	& \text{if $ j\in S(i) $ and $ m=b_{j}+1 $}, \\
\mathbb{0},	& \text{otherwise}.
\end{cases}
$$

Consider $ M=\max\{M_{1},M_{2}\} $. If $ M_{2}>M_{1} $, then put $ G_{m}=\mathbb{0} $ for all $ m=M_{1}+1,M_{1}+2,\ldots,M_{2} $.

We have the following result \cite{Krivulin1995Amax-algebra,Krivulin1996Max-plus,Krivulin2001Algebraic}.
\begin{lemma}
Suppose that the graph $ \mathcal{G}_{0} $ associated with the matrix $ G_{0} $ is acyclic, and $ r $ is the length of its longest path. Then the dynamics of the network is described in the semiring $ \mathbb{R}_{\max,+} $ by the equation \eqref{E-xk-Exp} with the matrices defined under manufacturing blocking rule as
\begin{align*}
A_{1}(k)
&=
(I\oplus T_{k}G_{0}^{T})^{r}( T_{k}(I\oplus G_{1}^{T})\oplus H_{1}), \\
A_{m}(k)
&=
(I\oplus T_{k}G_{0}^{T})^{r}( T_{k}G_{m}^{T}\oplus H_{m}),
\quad
m=2,\ldots,M;
\end{align*}
and under communication blocking rule as
\begin{align*}
A_{1}(k)
&=
(I\oplus T_{k}G_{0}^{T})^{r}T_{k}(I\oplus G_{1}^{T}\oplus H_{1}), \\
A_{m}(k)
&=
(I\oplus T_{k}G_{0}^{T})^{r}T_{k}(G_{m}^{T}\oplus H_{m}),
\quad
m=2,\ldots,M.
\end{align*}
\end{lemma}

\section{The mean service cycle time in queueing systems}

Consider a queueing system and suppose that its dynamics is described by equation~\eqref{E-xkAkxk1}. The evolution of the system can be represented as a sequence of service cycles. The first cycle starts at the initial time, and it is terminated as soon as all the servers in the network complete their first service, the second cycle is terminated as soon as the servers complete their second service, and so on. Clearly, the completion time of the $k$th cycle can be represented as $ \|\bm{x}(k)\| $. With the condition $ \bm{x}(0)=0 $, we have
$$
\|\bm{x}(k)\|
=
\|A_{k}\|,
\qquad
A_{k}
=
A(k)\cdots A(1).
$$

In the queueing systems, the mean growth rate of the state vector (Lyapunov exponent) can be regarded as the mean service cycle time. The reciprocal of the mean growth rate can be considered as the system throughput.

Below we show how  for some queueing networks, the mean service cycle time is evaluated based on examination of their related system matrices $ A(k) $.

\subsection{Open and closed tandem queues}

Tandem queueing systems present networks with the simplest topology determined by graphs which include only the nodes with no more than one incoming and outgoing arcs. Consider an open system with $ n $ nodes and infinite buffers, depicted in Figure~\ref{F-OTQS}.
\begin{figure}[ht]
\begin{center}
\begin{picture}(110,20)
\newsavebox\queue
\savebox{\queue}(10,6){\thicklines
 \put(0,3){\line(1,0){6}}
 \put(0,-3){\line(1,0){6}}
 \put(6,3){\line(0,-1){6}}
 \put(9,0){\circle{6}}}

\put(11,16){$1$}
\put(0,0){$c_{1}=\infty$}
\put(3,7){\usebox\queue}
\put(15,10){\vector(1,0){15}}

\put(38,16){$2$}
\put(27,0){$c_{2}=0$}
\put(30,7){\usebox\queue}
\put(42,10){\vector(1,0){15}}

\multiput(60,10)(2,0){3}{\circle*{1}}
\put(67,10){\vector(1,0){15}}

\put(90,16){$n$}
\put(79,0){$c_{n}=0$}
\put(82,7){\usebox\queue}
\put(94,10){\vector(1,0){15}}

\end{picture}
\end{center}
\caption{Open tandem queues with infinite buffers.}\label{F-OTQS}
\end{figure}
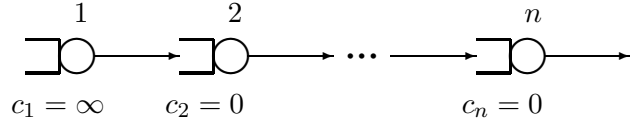

Since $ M=0 $, we put $ G_{1}=\mathbb{0} $ and then get equation \eqref{E-xkAkxk1} with the matrix
$$
A(k)
=
(I\oplus T_{k}G_{0}^{T})^{n-1}T_{k}
=
\left(
	\begin{array}{cccc}
		\tau_{1k}									& \mathbb{0}								& \ldots	& \mathbb{0} \\
		\tau_{1k}\tau_{2k}				& \tau_{2k}									& 				& \mathbb{0} \\
		\vdots										& \vdots										& \ddots	& \\
		\tau_{1k}\cdots\tau_{nk}	& \tau_{2k}\cdots\tau_{nk}	& \ldots	& \tau_{nk}
	\end{array}
\right).
$$

The state transition matrix is triangular. The application of Theorem~\ref{T-lambdaTr} yields
$$
\lambda
=
\mathop\mathrm{tr}(\mathsf{E}T_{1})
=
\max(\mathsf{E}\tau_{11},\ldots,\mathsf{E}\tau_{n1}).
$$

Consider a closed tandem queueing system presented in Figure~\ref{F-CTQS}.
\begin{figure}[ht]
\begin{center}
\begin{picture}(120,30)
\savebox{\queue}(10,6){\thicklines
 \put(0,3){\line(1,0){6}}
 \put(0,-3){\line(1,0){6}}
 \put(6,3){\line(0,-1){6}}
 \put(9,0){\circle{6}}}

\put(23,24){$1$}
\put(16,8){$c_{1}$}
\put(15,15){\usebox\queue}
\put(27,18){\vector(1,0){15}}

\put(50,24){$2$}
\put(43,8){$c_{2}$}
\put(42,15){\usebox\queue}
\put(54,18){\vector(1,0){15}}

\multiput(72,18)(2,0){3}{\circle*{1}}
\put(79,18){\vector(1,0){15}}

\put(102,24){$n$}
\put(95,8){$c_{n}$}
\put(94,15){\usebox\queue}
\put(106,18){\vector(1,0){15}}

\put(121,18){\line(0,-1){18}}
\put(121,0){\line(-1,0){121}}
\put(0,0){\line(0,1){18}}
\put(0,18){\vector(1,0){15}}

\end{picture}
\end{center}
\caption{A closed tandem queueing system with infinite buffers.}\label{F-CTQS}
\end{figure}
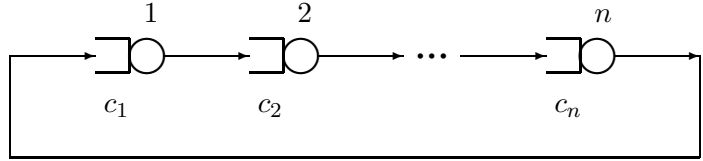

Suppose $ c_{i}=1 $ for all $ i=1,\ldots,n $. Then we have $ M=1 $. The state transition matrix takes the form
$$
A(k)
=
T_{k}(I\oplus G_{1}^{T})
=
\left(
	\begin{array}{ccccc}
		\tau_{1k}		& \mathbb{0}	& \ldots		& \mathbb{0}	& \tau_{1k} \\
		\tau_{2k}		& \tau_{2k}		& 					& \mathbb{0}	& \mathbb{0} \\
								& \ddots			& \ddots		& 						& \\
								& 						& \ddots		& \ddots			& \\
		\mathbb{0}	& \mathbb{0}	& 					& \tau_{nk}		& \tau_{nk}
	\end{array}
\right).
$$

With $ n=2 $, the matrix $ A(k) $ is a similarity matrix with the factor $ \|A(k)\|=\tau_{1k}\oplus\tau_{2k}=\mathop\mathrm{tr}T_{k} $. Therefore, we have
$$
\lambda
=
\mathsf{E}\mathop\mathrm{tr}T_{1}
=
\mathsf{E}\max(\tau_{11},\tau_{21}).
$$

\subsection{Tandem queues with finite buffers and blocking}

Consider an open tandem queueing system which may have finite buffers (Figure~\ref{F-OTQSB}).
\begin{figure}[ht]
\begin{center}
\begin{picture}(110,30)
\savebox{\queue}(10,6){\thicklines
 \put(0,3){\line(1,0){6}}
 \put(0,-3){\line(1,0){6}}
 \put(6,3){\line(0,-1){6}}
 \put(9,0){\circle{6}}}

\newsavebox\queuef
\savebox{\queuef}(10,6){\thicklines
 \put(0,3){\line(1,0){6}}
 \put(0,-3){\line(1,0){6}}
 \put(6,3){\line(0,-1){6}}
 \put(0,3){\line(0,-1){6}}
 \put(9,0){\circle{6}}}

\put(11,24){$1$}
\put(0,7){$c_{1}=\infty$}
\put(0,0){$b_{1}=\infty$}
\put(3,15){\usebox\queue}
\put(15,18){\vector(1,0){15}}

\put(38,24){$2$}
\put(27,7){$c_{2}=0$}
\put(34,0){$b_{2}$}
\put(30,15){\usebox\queuef}
\put(42,18){\vector(1,0){15}}

\multiput(60,18)(2,0){3}{\circle*{1}}
\put(67,18){\vector(1,0){15}}

\put(90,24){$n$}
\put(79,7){$c_{n}=0$}
\put(86,0){$b_{n}$}
\put(82,15){\usebox\queuef}
\put(94,18){\vector(1,0){15}}

\end{picture}
\end{center}
\caption{Open tandem queues with finite buffers.}\label{F-OTQSB}
\end{figure}
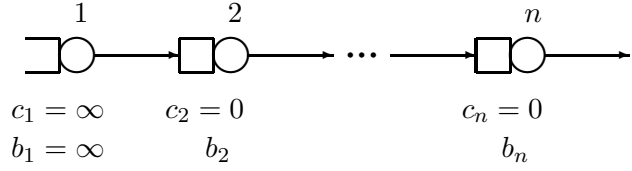

Suppose that the system operates under the manufacturing blocking rule. Let $ n=3 $, $ b_{1}=b_{2}=\infty $, and $ b_{3}=0 $. Then we have $ M=1 $. Evaluation of the system matrix gives
$$
A(k)
=
(I\oplus T_{k}G_{0}^{T})^{2}( T_{k}\oplus H_{1})
=
\left(
	\begin{array}{ccccc}
		\tau_{1k}										& \mathbb{0}					& \mathbb{0} \\
		\tau_{1k}\tau_{2k}					& \tau_{2k}						& \mathbb{1} \\
		\tau_{1k}\tau_{2k}\tau_{3k}	& \tau_{2k}\tau_{3k}	& \tau_{3k}
	\end{array}
\right).
$$

Taking into account collinearity of the last two rows, we get the skeleton decomposition
\begin{multline*}
A(k)
=
\left(
	\begin{array}{ccccc}
		\tau_{1k}										& \mathbb{0}					& \mathbb{0} \\
		\tau_{1k}\tau_{2k}					& \tau_{2k}						& \mathbb{1} \\
		\tau_{1k}\tau_{2k}\tau_{3k}	& \tau_{2k}\tau_{3k}	& \tau_{3k}
	\end{array}
\right)
\\
=
\left(
	\begin{array}{cc}
		\mathbb{1}	& \mathbb{0} \\
		\mathbb{0}	& \mathbb{1} \\
		\mathbb{0}	& \tau_{3k}
	\end{array}
\right)
\left(
	\begin{array}{ccc}
		\tau_{1k}										& \mathbb{0}					& \mathbb{0} \\
		\tau_{1k}\tau_{2k}					& \tau_{2k}						& \mathbb{1}
	\end{array}
\right)
=
B(k)C(k).
\end{multline*}

Consider the matrix
$$
C(k)B(k+1)
=
\left(
	\begin{array}{cc}
		\tau_{1k}										& \mathbb{0} \\
		\tau_{1k}\tau_{2k}					& \tau_{2k}\oplus\tau_{3,k+1}
	\end{array}
\right).
$$

By application of Theorem~\ref{T-lambdaC1B2}, we conclude that
$$
\lambda
=
\mathop\mathrm{tr}(\mathsf{E}[C(1)B(1)])
=
\max(\mathsf{E}\tau_{11},\mathsf{E}\max(\tau_{21},\tau_{31})).
$$

Assume that the system in Figure~\ref{F-OTQSB} follows the communication blocking rule. With $ n=3 $, $ b_{1}=\infty $, and $ b_{2}=b_{3}=0 $, we have $ M=1 $, $ G_{1}=\mathbb{0} $, and $ H_{1}=G_{0} $. The state transition matrix is represented as
$$
A(k)
=
(I\oplus T_{k}G_{0}^{T})^{2}T_{k}(I\oplus G_{0})
=
\left(
	\begin{array}{ccc}
		\tau_{1k}										& \tau_{1k}										& \mathbb{0} \\
		\tau_{1k}\tau_{2k}					& \tau_{1k}\tau_{2k}					& \tau_{2k} \\
		\tau_{1k}\tau_{2k}\tau_{3k}	& \tau_{1k}\tau_{2k}\tau_{3k}	& \tau_{2k}\tau_{3k}
	\end{array}
\right).
$$

Furthermore, we get the decomposition
$$
A(k)
=
\left(
	\begin{array}{cc}
		\mathbb{1}					& \mathbb{0} \\
		\tau_{2k}						& \tau_{2k} \\
		\tau_{2k}\tau_{3k}	& \tau_{2k}\tau_{3k}
	\end{array}
\right)
\left(
	\begin{array}{ccc}
		\tau_{1k}		& \tau_{1k}		& \mathbb{0} \\
		\mathbb{0}	& \mathbb{0}	& \mathbb{1}
	\end{array}
\right)
=
B(k)C(k).
$$

Consider the matrix product
\begin{multline*}
C(k)B(k+1)
=
\left(
	\begin{array}{cc}
		\tau_{1k}\tau_{2,k+1}	& \tau_{1k}\tau_{2,k+1} \\
		\tau_{2,k+1}\tau_{3,k+1}	& \tau_{2,k+1}\tau_{3,k+1}
	\end{array}
\right)
\\
=
\left(
	\begin{array}{c}
		\tau_{1k} \\
		\tau_{3,k+1}
	\end{array}
\right)
\left(
	\begin{array}{cc}
		\tau_{2,k+1}	& \tau_{2,k+1}
	\end{array}
\right)
=
\bm{u}(k)\bm{v}^{T}(k).
\end{multline*}

Since the product is actually a matrix of rank $ 1 $, we finally get
$$
\lambda
=
\mathsf{E}[\bm{v}^{T}(1)\bm{u}(2)]
=
\mathsf{E}\tau_{21}+\mathsf{E}\max(\tau_{11},\tau_{31}).
$$

\subsection{A fork-join network}

Now we turn to network models which involve fork and join operations. Consider a fork-join network with $ n=5 $ nodes and infinite buffers, depicted in Figure~\ref{F-AFJQN} \cite{Krivulin1996Max-plus}.
\begin{figure}[ht]
\begin{center}
\begin{picture}(110,60)
\savebox{\queue}(10,6){\thicklines
 \put(0,3){\line(1,0){6}}
 \put(0,-3){\line(1,0){6}}
 \put(6,3){\line(0,-1){6}}
 \put(9,0){\circle{6}}}

\put(11,32){$1$}
\put(0,16){$c_{1}=\infty$}
\put(3,23){\usebox\queue}
\put(15,26){\vector(1,0){15}}

\put(38,32){$2$}
\put(27,16){$c_{2}=0$}
\put(30,23){\usebox\queue}
\put(42,26){\vector(1,1){15}}
\put(42,26){\vector(1,-1){15}}

\put(65,47){$3$}
\put(54,31){$c_{3}=1$}
\put(57,38){\usebox\queue}
\put(69,41){\line(1,1){10}}
\put(79,51){\line(0,1){8}}
\put(79,59){\line(-1,0){59}}
\put(20,59){\line(0,-1){23}}
\put(20,36){\vector(1,-1){10}}
\put(69,41){\vector(1,-1){15}}

\put(65,17){$4$}
\put(54,1){$c_{4}=0$}
\put(57,8){\usebox\queue}
\put(69,11){\vector(1,1){15}}

\put(92,32){$5$}
\put(81,16){$c_{5}=1$}
\put(84,23){\usebox\queue}
\put(96,26){\vector(1,0){14}}

\end{picture}
\end{center}
\caption{A fork-join queueing network with infinite buffers.}\label{F-AFJQN}
\end{figure}
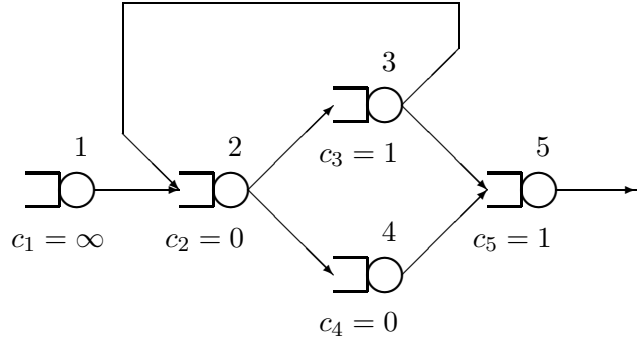

Since $ r=2 $, we arrive at equation \eqref{E-xkAkxk1} with the matrix
\begin{multline*}
A(k)
=
(I\oplus T_{k}G_{0}^{T})^{2}T_{k}(I\oplus G_{1}^{T})
\\
=
\left(
	\begin{array}{ccccc}
		\tau_{1k}										& \mathbb{0}									& \mathbb{0}									& \mathbb{0}	& \mathbb{0} \\
		\tau_{1k}\tau_{2k}					& \tau_{2k}\tau_{3k}					& \tau_{2k}\tau_{3k}					& \mathbb{0}	& \mathbb{0} \\
		\mathbb{0}									& \tau_{3k}										& \tau_{3k}										& \mathbb{0}	& \mathbb{0} \\
		\tau_{1k}\tau_{2k}\tau_{4k}	& \tau_{2k}\tau_{3k}\tau_{4k}	& \tau_{2k}\tau_{3k}\tau_{4k}	& \tau_{4k}		& \mathbb{0} \\
		\mathbb{0}									& \mathbb{0}									& \tau_{5k}										& \tau_{5k}		& \tau_{5k}
	\end{array}
\right).
\end{multline*}

The matrix can be represented as
\begin{multline*}
A(k)
=
\left(
	\begin{array}{cccc}
		\mathbb{1}					& \mathbb{0}					& \mathbb{0}	& \mathbb{0} \\
		\tau_{2k}						& \tau_{2k}						& \mathbb{0}	& \mathbb{0} \\
		\mathbb{0}					& \mathbb{1}					& \mathbb{0}	& \mathbb{0} \\
		\tau_{2k}\tau_{4k}	& \tau_{2k}\tau_{4k}	& \tau_{4k}		& \mathbb{0} \\
		\mathbb{0}					& \mathbb{0}					& \tau_{5k}		& \tau_{5k}
	\end{array}
\right)
\left(
	\begin{array}{ccccc}
		\tau_{1k}	& \mathbb{0}		& \mathbb{0}	& \mathbb{0}	& \mathbb{0} \\
		\mathbb{0}	& \tau_{3k}		& \tau_{3k}		& \mathbb{0}	& \mathbb{0} \\
		\mathbb{0}	& \mathbb{0}	& \mathbb{1}	& \mathbb{1}	& \mathbb{0} \\
		\mathbb{0}	& \mathbb{0}	& \mathbb{1}	& \mathbb{0}	& \mathbb{1}
	\end{array}
\right)
\\
=
B(k)C(k).
\end{multline*}

Let us examine the matrix
$$
C(k)B(k+1)
=
\left(
	\begin{array}{cccc}
		\tau_{1k}									& \mathbb{0}								& \mathbb{0}		& \mathbb{0} \\
		\tau_{2,k+1}\tau_{3k}			& \tau_{2,k+1}\tau_{3k}			& \mathbb{0}		& \mathbb{0} \\
		\tau_{2,k+1}\tau_{4,k+1}	& \tau_{2,k+1}\tau_{4,k+1}	& \tau_{4,k+1}	& \mathbb{0} \\
		\mathbb{0}								& \mathbb{1}								& \tau_{5,k+1}	& \tau_{5,k+1}
	\end{array}
\right).
$$

Taking into account that the matrix is triangular, we apply Theorem~\ref{T-lambdaC1B2} to conclude that
$$
\lambda
=
\mathop\mathrm{tr}(\mathsf{E}[C(1)B(1)])
=
\max(\mathsf{E}\tau_{11},\mathsf{E}\tau_{21}+\mathsf{E}\tau_{31},\mathsf{E}\tau_{41},\mathsf{E}\tau_{51}).
$$

\subsection{A system with round routing}

Consider an open system depicted in Figure~\ref{F-RRQS}, which consists of $ n=l+1 $ queues labeled with $ 0,1,\ldots,l $. Queue $ 0 $ is intended to represent an external arrival stream of customers. Each incoming
customer has to go to one of the other queues, being chosen by a regular round routing mechanism, and then leaves the system.

With the routing mechanism, the customer that is the first to depart queue $ 0 $ goes to the $1$st queue, the second customer does to the $2$nd queue, and so on. After the $l$th customer directed to queue $ l $, the next $(l+1)$st customer is directed to the $1$st queue once again, and the procedure is further repeated round and round.
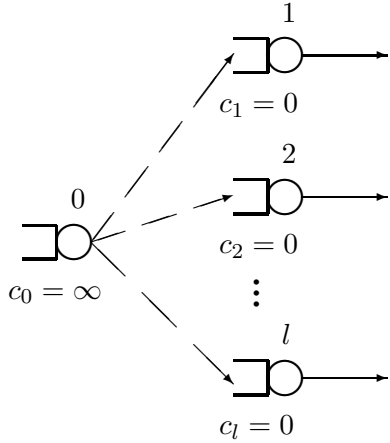
\begin{figure}[ht]
\begin{center}
\begin{picture}(70,75)
\savebox{\queue}(10,6){\thicklines
 \put(0,3){\line(1,0){6}}
 \put(0,-3){\line(1,0){6}}
 \put(6,3){\line(0,-1){6}}
 \put(9,0){\circle{6}}}

\put(11,40){$0$}
\put(0,24){$c_{0}=\infty$}
\put(3,31){\usebox\queue}
\multiput(15,34)(6,8){4}{\line(3,4){4}}
\put(36,62){\vector(3,4){4}}

\multiput(15,34)(9,3){3}{\line(3,1){6}}
\put(36,41){\vector(3,1){4}}

\multiput(15,34)(9,-9){2}{\line(1,-1){6}}
\put(33,16){\vector(1,-1){7}}

\put(48,73){$1$}
\put(37,57){$c_{1}=0$}
\put(40,64){\usebox\queue}
\put(52,67){\vector(1,0){15}}

\put(48,48){$2$}
\put(37,32){$c_{2}=0$}
\put(40,39){\usebox\queue}
\put(52,42){\vector(1,0){15}}

\multiput(44,23)(0,2){3}{\circle*{1}}

\put(48,16){$l$}
\put(37,0){$c_{l}=0$}
\put(40,7){\usebox\queue}
\put(52,10){\vector(1,0){15}}

\end{picture}
\end{center}
\caption{A system with round robin routing.}\label{F-RRQS}
\end{figure}

The above system can be replaced with an equivalent fork-join network \cite{Krivulin1996Analgebraic} which consists of $ n=2l $ nodes (see Figure~\ref{F-RRFJQN}) provided that for every nodes $ i=l+1,l+2,\ldots,2l $ in the new system, the service time is defined as
$$
\tau_{ik}
=
\tau_{0,kl-2l+i}.
$$

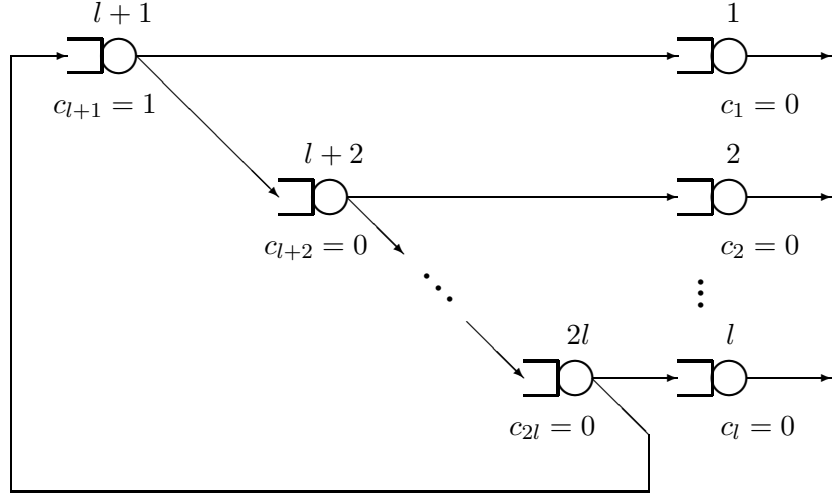
\begin{figure}[ht]
\begin{center}
\begin{picture}(145,85)
\savebox{\queue}(10,6){\thicklines
 \put(0,3){\line(1,0){6}}
 \put(0,-3){\line(1,0){6}}
 \put(6,3){\line(0,-1){6}}
 \put(9,0){\circle{6}}}

\put(14,83){$l+1$}
\put(7,67){$c_{l+1}=1$}
\put(10,74){\usebox\queue}
\put(22,77){\vector(1,0){95}}
\put(22,77){\vector(1,-1){25}}

\put(125,83){$1$}
\put(124,67){$c_{1}=0$}
\put(117,74){\usebox\queue}
\put(129,77){\vector(1,0){15}}

\put(51,58){$l+2$}
\put(44,42){$c_{l+2}=0$}
\put(47,49){\usebox\queue}
\put(59,52){\vector(1,0){58}}
\put(59,52){\vector(1,-1){10}}

\put(125,58){$2$}
\put(124,42){$c_{2}=0$}
\put(117,49){\usebox\queue}
\put(129,52){\vector(1,0){15}}

\put(80,30){\vector(1,-1){10}}

\put(97,26){$2l$}
\put(87,10){$c_{2l}=0$}
\put(90,17){\usebox\queue}
\put(102,20){\vector(1,0){15}}
\put(102,20){\line(1,-1){10}}
\put(112,10){\line(0,-1){10}}
\put(112,0){\line(-1,0){112}}

\put(0,0){\line(0,1){77}}
\put(0,77){\vector(1,0){10}}

\put(125,26){$l$}
\put(124,10){$c_{l}=0$}
\put(117,17){\usebox\queue}
\put(129,20){\vector(1,0){15}}

\multiput(73,38)(2,-2){3}{\circle*{1}}

\multiput(121,33)(0,2){3}{\circle*{1}}

\end{picture}
\end{center}
\caption{An equivalent fork-join network.}\label{F-RRFJQN}
\end{figure}

Suppose that $ l=2 $. Evaluation of the system matrix gives
$$
A(k)
=
(I\oplus T_{k}G_{0}^{T})^{2}T_{k}(I\oplus G_{1}^{T})
=
\left(
	\begin{array}{cccc}
		\tau_{1k}		& \mathbb{0}	& \tau_{1k}\tau_{3k}					& \tau_{1k}\tau_{3k} \\
		\mathbb{0}	& \tau_{2k}		& \tau_{2k}\tau_{3k}\tau_{4k}	& \tau_{2k}\tau_{3k}\tau_{4k} \\
		\mathbb{0}	& \mathbb{0}	& \tau_{3k}										& \tau_{3k} \\
		\mathbb{0}	& \mathbb{0}	& \tau_{3k}\tau_{4k}					& \tau_{3k}\tau_{4k}
	\end{array}
\right).
$$

Let us represent $ A(x) $ in the form
$$
A(k)
=
\left(
	\begin{array}{ccc}
		\tau_{1k}		& \mathbb{0}	& \tau_{1k} \\
		\mathbb{0}	& \tau_{2k}		& \tau_{2k}\tau_{4k} \\
		\mathbb{0}	& \mathbb{0}	& \mathbb{1} \\
		\mathbb{0}	& \mathbb{0}	& \tau_{4k}
	\end{array}
\right)
\left(
	\begin{array}{cccc}
		\mathbb{1}	& \mathbb{0}	& \mathbb{0}	& \mathbb{0} \\
		\mathbb{0}	& \mathbb{1}	& \mathbb{0}	& \mathbb{0} \\
		\mathbb{0}	& \mathbb{0}	& \tau_{3k}		& \tau_{3k}
	\end{array}
\right)
=
B(k)C(k).
$$

Since the matrix product
$$
C(k)B(k+1)
=
\left(
	\begin{array}{ccc}
		\tau_{1,k+1}	& \mathbb{0}		& \tau_{1,k+1} \\
		\mathbb{0}		& \tau_{2,k+1}	& \tau_{2,k+1}\tau_{4,k+1} \\
		\mathbb{0}		& \mathbb{0}		& \tau_{3k}\tau_{4,k+1}
	\end{array}
\right)
$$
has a triangular form, we get
$$
\lambda
=
\mathop\mathrm{tr}(\mathsf{E}[C(1)B(2)])
=
\max(\mathsf{E}\tau_{11},\mathsf{E}\tau_{21},\mathsf{E}\tau_{31}+\mathsf{E}\tau_{41}).
$$

Considering that $ \mathsf{E}\tau_{31}=\mathsf{E}\tau_{41}=\mathsf{E}\tau_{01} $, we finally have
$$
\lambda
=
\max(2\mathsf{E}\tau_{01},\mathsf{E}\tau_{11},\mathsf{E}\tau_{21}).
$$

\section{Conclusion}

A stochastic dynamical system governed by the vector equation which is linear in some idempotent semiring was considered. New general conditions for Lyapunov exponent to exist for the system were given which involve the spectral radius of the mean state transition matrix. New method of evaluation of the exponent was proposed based on a decomposition of the system state transition matrix.

The above method was applied to the analysis of a class of queueing systems including open and closed tandem queues with finite and infinite buffers, fork-join networks, and systems with round-robin routing. Examples of evaluation of the mean service cycle time for the systems were considered, and related results in the form of some functions of the mean values of random variables that determine the service time were presented.

\bibliographystyle{utphys}

\bibliography{Evaluation_of_Lyapunov_exponent_in_generalized_linear_dynamical_models_of_queueing_networks}

\end{document}